\documentclass[10pt]{article}

\usepackage{amsmath,amsfonts, latexsym,graphicx, footmisc, fancyhdr}

 \usepackage{hyperref}

\hypersetup{colorlinks=true, urlcolor=blue, citecolor=blue, linkcolor=blue}

\newcommand{\U}{{\mathcal U}}
\newcommand{\0}{{\mathbf 0}}
\newcommand{\C}{{\mathbb C}}
\newcommand{\Z}{{\mathbb Z}}

\newcommand{\N}{{\mathbb N}}
\newcommand{\cL}{{\mathbb L}}
\newcommand{\D}{{\mathbb D}}

\newcommand{\strat}{{\mathfrak S}}

\newcommand{\hyp}{{\mathbb H}}
\newcommand{\supp}{\operatorname{supp}}

\newcommand{\arrow}[1]{\stackrel{#1}{\longrightarrow}}
\newcommand{\Adot}{\mathbf A^\bullet}

\newcommand{\Pdot}{\mathbf P^\bullet}

\newcommand{\vd}{{{\cal D}}}

\newtheorem{defn0}{Definition}[section]
\newtheorem{prop0}[defn0]{Proposition}
\newtheorem{conj0}[defn0]{Conjecture}
\newtheorem{thm0}[defn0]{Theorem}
\newtheorem{lem0}[defn0]{Lemma}
\newtheorem{corollary0}[defn0]{Corollary}
\newtheorem{example0}[defn0]{Example}
\newtheorem{remark0}[defn0]{Remark}
\newtheorem{question0}[defn0]{Question}

\newenvironment{defn}{\begin{defn0}\hskip -.06in .}{\end{defn0}}
\newenvironment{prop}{\begin{prop0}\hskip -.06in .}{\end{prop0}}

\newenvironment{thm}{\begin{thm0}\hskip -.06in .}{\end{thm0}}

\newenvironment{cor}{\begin{corollary0}\hskip -.06in .}{\end{corollary0}}
\newenvironment{exm}{\begin{example0}\hskip -.06in .\rm}{\end{example0}}
\newenvironment{rem}{\begin{remark0}\hskip -.06in .\rm}{\end{remark0}}

\newcommand{\thmref}[1]{Theorem~\ref{#1}}

\newcommand{\corref}[1]{Corollary~\ref{#1}}

\newcommand{\secref}[1]{Section~\ref{#1}}

\newcommand{\mbf}[1]{{\mathbf #1}}

\newcommand{\qed}{\mbox{$\Box$}}
\newenvironment{proof}{\noindent {\bf Proof.}}{\qed\vskip 6pt}

\title{Milnor fibers and Links of Local Complete Intersections}

\author{David B. Massey}

\date{}

\begin{document}

\baselineskip = 14pt

\maketitle

\thispagestyle{fancy}

\lfoot{AMS subject classifications: 32B15, 32C35, 32C18, 32B10.
}
\cfoot{}
\rfoot{}
\renewcommand{\footrulewidth}{0.4pt}

\begin{abstract} We discuss and prove a number of cohomological results for Milnor fibers, real links, and complex links of local complete intersections with singularities of arbitrary dimension.
\end{abstract}

\sloppy




\section{Introduction} 

Suppose that $X$ is a purely $(n+1)$-dimensional local complete intersection. For convenience, we will assume that $X$ is locally embedded in an open subset $\U$ of $\C^N$, and that $\0\in X$. We will assume that $\0$ is a singular point of $X$. Suppose that we have a complex analytic function $\hat f:(\U, \0)\rightarrow (\C, 0)$, and let $f$ denote the restriction of $\hat f$ to $X$. 

We are interested in what one can say about standard topological objects -- like real links, complex links, and Milnor fibers (see \secref{sec:classic} for definitions) - in the case where the dimension of the singular set of $X$, $\Sigma X$, and the dimension of the critical locus of $f$, $\Sigma f$, are arbitrary. Of course, part of the issue is that we must define what notion of ``critical locus'' we are using for a function on a space with possibly non-isolated singularities.

We let $d:=\dim_\0\Sigma X$. If $d=0$, we define $d_f:=0$; otherwise, we define $d_f:=\dim_\0\overline{\Sigma X\backslash V(f)}$. We are interested in points which are ``cohomological critical points of $f$'', that is, points where the reduced cohomology of the Milnor fiber is not zero. Therefore, we let $F_{f-f(\mbf x),\mbf x}$ denote the compact Milnor fiber of $f-f(\mbf x)$ at $\mbf x$, and define:

\begin{defn} The {\bf $\Z$-critical locus of $f$}, $\Sigma _{\Z}f$, is the set
$$
\left\{\mbf x\in X \ | \ \widetilde H^*\big(F_{f-f(\mbf x), \mbf x}; \Z\big)\neq 0\right\}.
$$
\end{defn}
It is well-known that, given any Whitney stratification of $X$, the $\Z$-critical locus is contained in the stratified critical locus of $f$. In addition, stratified critical values are local isolated. Consequently, by taking a smaller neighborhood of the origin, if necessary, we may assume that $\Sigma _{\Z}f\subseteq V(f)$. We let $s_f:=\dim_\0 \Sigma _{\Z}f$.

\smallskip

Milnor fibers and the real link of a hypersurface singularity have been fundamental objects of study since Milnor's book \cite{milnorsing} on the subject appeared in 1968. For complete intersections with isolated singular points, the Milnor fiber and the real link were investigated by Looijenga in his book \cite{looibook} from 1984.  The complex link was studied in depth by L\^e and Kato in 1975 in \cite{lk}, and for complete intersections by L\^e in 1979 in \cite{levan}. The complex link, and its relationship to the real link, was discussed at length in Goresky and MacPherson's 1988 book ``Stratified Morse Theory'' \cite{stratmorse}.

\smallskip

There are several known results in the general situation that we are in. Note that we always use $\Z$ coefficients in this paper when dealing with ordinary cohomology.

\begin{enumerate}

\item[$\bullet$] In \cite{hamm}, Hamm generalizes a result of Milnor for hypersurfaces and proves that the real link of $X$ at $\mbf 0$, $K_{X, \0}$, is $(n-1)$-connected, and so, in particular, that
$\widetilde H^k(K_{X, \0})=0$ for $k\leq n-1$.

\item[$\bullet$] In \cite{levan}, L\^e proves that the complex link of $X$ at $\mbf 0$, $\cL_{X, \mbf 0}$, has the homotopy-type of a finite bouquet of $n$-spheres and so, in particular, that
$\widetilde H^k(\cL_{X, \0})=0$ for $k\neq n$.

\item[$\bullet$] L\^e's result on the complex link easily extends to the links of all strata (see \secref{sec:classic}) of any Whitney stratification of $X$. This implies that the shifted constant sheaf $\Z^\bullet_X[n+1]$ is a {\it perverse sheaf}. Then, the general results that the shifted nearby and vanishing cycles of a perverse sheaf are perverse imply that $\widetilde H^k(F_{f,\0})=0$ unless $n-s_f\leq k\leq n$, and $H^k(F_{f,\0}, \partial F_{f,\0})=0$ unless $n\leq k\leq 2n$.

\item[$\bullet$] In Chapter 6 of \cite{schurbook}, Sch\"urmann investigates similar questions by similar techniques, but does not obtain most of our general results.

\end{enumerate}

Notice how none of the results above detect the singular set of $X$ itself; they are independent of the values of $d$ and $d_f$. We want to look at cohomological results that actually notice that $X$ may be very singular. 

\smallskip

In this paper, we prove:

\begin{enumerate}
\item $H^k(F_{f,\0}, \partial F_{f,\0})=0$ unless $k=2n$ or $n\leq k\leq n+\operatorname{max}\{s_f, d_f\}$.

\item $H^k(\cL_{X,\0}, \partial \cL_{X,\0})=0$ unless $k=2n$ or $n\leq k\leq n+d$.

\item There is a (dual) Wang exact sequence

$
\cdots\rightarrow H^{k-1}(F_{f,\0}, \partial F_{f,\0})\rightarrow H^k(K_{X,\0}, K_{V(f),\0})\rightarrow  H^{k}(F_{f,\0}, \partial F_{f,\0})\xrightarrow{\operatorname{id}-T_{f, \0}} \hfill
$

$
\hfill H^{k}(F_{f,\0}, \partial F_{f,\0})  \rightarrow H^{k+1}(K_{X,\0}, K_{V(f),\0})\rightarrow\cdots$,

\noindent where $T_{f,\0}$ is the monodromy automorphism on the Milnor fiber with boundary. In particular, $H^k(K_{X,\0}, K_{V(f),\0})=0$ unless $k=2n, 2n+1$ or $n\leq k\leq n+\operatorname{max}\{s_f, d_f\}+1$.
\item If $H$ is a generic hyperplane in the ambient space $\U$, then $H^k(K_{X,\0}, K_{X\cap H,\0})=0$ unless $k=2n, 2n+1$ or $n\leq k\leq n+d+1$.
\item Suppose that $n\geq 1$ and $s_f=0$. Then, there is well-defined variation map 
$$\operatorname{var}:H^n(F_{f,\0})\rightarrow H^n(F_{f,\0}, \partial F_{f,\0})
$$
and an exact sequence
$$
0\rightarrow H^n(K_{X,\0})\rightarrow H^n(F_{f,\0})\arrow{\operatorname{var}} H^n(F_{f,\0}, \partial F_{f,\0})\rightarrow H^{n+1}(K_{X,\0})\rightarrow 0.
$$

In particular, $H^n(K_{X,\0})$ injects into $H^n(F_{f,\0})$, and $\operatorname{var}$
  is an isomorphism if and only if $H^n(K_{X,\0})= H^{n+1}(K_{X,\0})=0$. Furthermore, for all $k\geq n+1$, $H^k(F_{f,\0}, \partial F_{f,\0})\cong H^{k+1}(K_{X,\0})$.
  
  \item Suppose that $n\geq 1$ and $s_f=0$. Then, for all $k\geq n+1$, $H^k(F_{f,\0}, \partial F_{f,\0})\cong H^{k+1}(K_{X,\0})$, and so is independent of $f$. In particular, for all $k\geq n+1$,
  $$H^k(F_{f,\0}, \partial F_{f,\0}) \ \cong  \ H^{k+1}(K_{X,\0})\ \cong \ H^k(\cL_{X,\0}, \partial \cL_{X,\0}).$$

  \item $H^n(K_{X,\0})$ injects into $H^n(\cL_{X,\0})$.
 \item $H^{k}(K_{X,\0})=0$, unless $k=0$, $k=2n+1$, or $n\leq k\leq n+d+1$.
\end{enumerate}

Most of the results above are known in the case where $d=0$, i.e., where $X$ is an isolated complete intersection singularity; in that case, the Milnor fibers, real links, and complex links are all smooth and one can use Poincar\'e-Lefschetz duality. However, when $d\geq 1$, the Milnor fiber, real link and complex link are all singular, and one cannot use Poincar\'e-Lefschetz duality. However, Verdier duality is still at our disposal, and leads to all of these new general results.

\section{Standard definitions and a classical result}\label{sec:classic}

We must discuss some basic background on Whitney stratifications, Milnor fibrations, real and complex links, the derived category of complexes of sheaves of $\Z$-modules, and perverse sheaves. General references on stratifications, links, and Milnor fibrations are \cite{stratmorse},  \cite{milnorsing}, \cite{dimcasing}, and \cite{relmono}. 

\smallskip

Suppose that $\U$ is an open subset of $\C^N$, that $X$ is a complex analytic subset of $\U$, and that $\mbf p\in X$. Suppose that $f:(X, \mbf p)\rightarrow (\C, 0)$ is a complex analytic function. 

\medskip

Then, the following definitions are standard:

\begin{defn}

Let $B_\epsilon(\mbf p)$ be a sufficiently small closed ball of radius $\epsilon>0$, with boundary $S_\epsilon(\mbf p)$. Let $L:\C^N\rightarrow\C$ be a generic affine linear form with $L(\mbf p)=0$. Finally, let $\eta\in \C$ be such that $0<|\eta|\ll \epsilon$.

Then, the homeomorphism-types of the following spaces are independent of the choices of $\epsilon$, $\eta$, and $L$:

\begin{enumerate}

\item the {\bf real link} of $X$ at $\mbf p$ is $K_{X,\mbf p}:=S_\epsilon(\mbf p)\cap X$; 

\item the (compact) {\bf Milnor fiber of $f$ at $\mbf p$} is $F_{f,\mbf p}:=B_\epsilon(\mbf p)\cap f^{-1}(\eta)$;

\item the {\bf boundary of the the Milnor fiber of $f$ at $\mbf p$} is $\partial F_{f,\mbf p}:=S_\epsilon(\mbf p)\cap f^{-1}(\eta)$;

\item the (compact) {\bf complex link of $X$ at $\mbf p$} is $\cL_{X, \mbf p}:=B_\epsilon(\mbf p)\cap X\cap  L^{-1}(\eta)=F_{L_{|_X}, \mbf p}$; and 

\item the {\bf boundary of the complex link of $X$ at $\mbf p$} is $\partial \cL_{X, \mbf p}:=S_\epsilon(\mbf p)\cap X\cap L^{-1}(\eta)=\partial F_{L_{|_X}, \mbf p}$. 

\end{enumerate}

If we replace the closed ball with the open ball, the homotopy-types of the resulting non-compact Milnor fiber and complex link are the same as those of the compact ones.

Furthermore, if we let $H:=V(L)$, then $\partial \cL_{X, \mbf p}$ is homeomorphic to $S_\epsilon(\mbf p)\cap X \cap H = K_{X\cap H, \mbf p}$.
\end{defn}

\medskip

We also need to define complex links of Whitney strata. Throughout this paper, our Whitney strata are always assumed to be connected.

\begin{defn} Suppose that $\strat$ is a complex analytic Whitney stratification of $X$. Let $S\in\strat$, let $\mbf p\in S$, and let $M$ be a complex submanifold of $\U$ of complementary dimension (i.e., of dimension $N-\dim S$), which transversely intersects $S$ at $\mbf p$.

Then, $M$ is a {\bf normal slice} to $S$ at $\mbf p$ and, for sufficiently small $\epsilon>0$, the homeomorphism-type of the pair 
$$(B_\epsilon(\mbf p)\cap X\cap M, \ \cL_{X\cap M, \mbf p})$$ 
is independent of the choice of $M$ and $\mbf p$. 

This homeomorphism-type is the {\bf normal Morse data of $S$ in $X$} and the homeomorphism-type of $\cL_{X\cap M, \mbf p}$  is the {\bf complex link of $S$ in $X$}. We denote this pair by $(\N_{X,S}, \cL_{X, S})$.
\end{defn}

We need the fundamental result of L\^e from \cite{levan} on local complete intersections.

\begin{thm}\label{lelink} \textnormal{(L\^e, 1979)} Suppose that $X$ is a local complete intersection of pure dimension $(n+1)$ and that $\strat$ is a Whitney stratification of $X$. Let $S\in\strat$. 

Then, $\cL_{X, S}$ has the homotopy-type of a finite bouquet of spheres of dimension $n-\dim S$. In particular, $\widetilde H^k(\cL_{X, S};\Z)$ is zero if $k\neq n-\dim S$ and 
$$
\widetilde H^{n-\dim S}(\cL_{X, S};\Z) \ \cong \ \Z^m,
$$
for some finite value of $m$ (possibly $0$).
\end{thm}

\section{Background on the derived category} 

We continue as before with $\U$ being an open subset of $\C^N$, and $X$ being a complex analytic subset of $\U$. We now want to see what we get by using the machinery of the derived category, perverse sheaves, nearby cycles, and vanishing cycles. As general references for the derived category and perverse sheaves, we recommend \cite{kashsch} and \cite{dimcasheaves}.

\medskip

As we are interested in results on integral cohomology, throughout this paper, our complexes of sheaves $\Adot$ will be bounded, constructible sheaves of $\Z$-modules on $X$; we write $\Adot\in D^b_c(X)$.

\medskip

\noindent\rule{1in}{1pt}

\noindent {\bf Morse modules and perverse sheaves}

\medskip

We first need a cohomological version of a normal Morse data to strata with coefficients in a complex of sheaves $\Adot$.

\begin{defn}
Suppose that $\Adot$ is a bounded complex of sheaves of $\Z$-modules, which is constructible with respect to a Whitney stratification $\strat$. Let $S\in\strat$.

Then, using our notation from the previous section, the isomorphism-type of the hypercohomology
$$
m_S^k(\Adot) \ := \ \hyp^{k-\dim S}\big(\N_{X,S}, \cL_{X, S}; \ \Adot\big)
$$
is well-defined (i.e., is independent of the choices made in defining the normal Morse data). We call this the {\bf degree $k$ Morse module of $S$ with respect to $\Adot$}. The Morse modules are necessarily finitely-generated $\Z$-modules.

In particular, 
$$
m_S^k\left(\Z_X^\bullet[n+1]\right) \ = \ \hyp^{k-\dim S}\big(\N_{X,S}, \cL_{X, S}; \ \Z_X^\bullet[n+1]\big) \ \cong$$
$$
H^{n+1+k-\dim S}\big(\N_{X,S}, \cL_{X, S}; \ \Z\big) \ \cong \ \widetilde H^{n+k-\dim S}\big(\cL_{X, S}; \ \Z\big)
$$
\end{defn}

\bigskip

We do not wish to define perverse sheaves here, and we are not going to list all of the well-known results on perverse sheaves that we use in this section. However, we do want to state a fundamental result gives a characterization of perverse sheaves, a result which may not be so well-known.

\begin{thm} \textnormal{(\cite{kashsch}, Theorem 10.3.2)}\label{pureperverse} A bounded constructible complex of sheaves $\Pdot$ is perverse if and only if all of the Morse modules of $\Pdot$ are concentrated in degree zero, i.e., are zero in all non-zero degrees.

This is independent of the Whitney stratification, with respect to which $\Pdot$ is constructible, that is used in defining the Morse modules.
\end{thm}

In light of \thmref{pureperverse}, the result of L\^e in \thmref{lelink} tells us:

\begin{thm} Suppose that $X$ is a local complete intersection of pure dimension $(n+1)$. Then, $\Z^\bullet_X[n+1]$ is a perverse sheaf.
\end{thm}

\medskip

\noindent\rule{1in}{1pt}

\noindent{\bf Stalk and costalk cohomology}

\medskip

We are interested in critical points of complex analytic functions $f:X\rightarrow\C$. Let $j_f$ denote the inclusion of $V(f)$ into $X$. For all $\mbf x\in X$, let $m_{\mbf x}$ denote the inclusion of $\{\mbf x\}$ into $X$ and, for all $\mbf x\in V(f)$, let ${\hat m}_{\mbf x}$ denote the inclusion of $\{\mbf x\}$ into $V(f)$, so that $m_{\mbf x}=j_f\circ {\hat m}_{\mbf x}$. For $\Adot\in\D^b_c(X)$, the {\it stalk cohomology} of $\Adot$ at $x$ is
$$
H^k(\Adot)_{\mbf x} \ = \ H^k(m^*_{\mbf x}\Adot) \ \cong \ \hyp^k(B_\epsilon(\mbf x); \ \Adot),
$$
for sufficiently small $\epsilon>0$ ; the {\it costalk cohomology} of $\Adot$ at $x$ is
$$
H^k(m^!_{\mbf x}\Adot) \ \cong \ \hyp^k(B_\epsilon(\mbf x), B_\epsilon(\mbf x)-\{\mbf x\}; \ \Adot) \ \cong \ \hyp^k(B_\epsilon(\mbf x), S_\epsilon(\mbf x); \ \Adot),
$$
for sufficiently small $\epsilon>0$.

If $\Pdot$ is a perverse sheaf on a space $X$, which is supported on an $r$-dimensional subset of $X$, and $\mbf x\in X$, then:
\begin{enumerate}
\item $H^k(m^*_{\mbf x}\Adot)=0$, unless $-r\leq k\leq 0$, and
\item $H^k(m^!_{\mbf x}\Adot)=0$, unless $0\leq k\leq r$.
\end{enumerate}

\medskip

\noindent\rule{1in}{1pt}

\noindent{\bf Nearby and vanishing cycles, the Milnor fiber, and the monodromy}

\medskip

We remind the reader that that the shifted nearby and vanishing cycle functors, $\psi_f[-1]$ and $\phi_f[-1]$, take perverse sheaves on $X$ to perverse sheaves on $V(f)$, and their stalk cohomology gives the cohomology and ``reduced'' cohomology of the Milnor of $f$. To be precise,
$$
H^k(\psi_f[-1]\Adot)_{\mbf x} \ \cong \ H^k({\hat m}_{\mbf x}\psi_f[-1]\Adot)\ \cong \ \hyp^{k-1}(F_{f, \mbf x};\ \Adot)
$$
and
$$
H^k(\phi_f[-1]\Adot)_{\mbf x} \ \cong \ H^k({\hat m}_{\mbf x}\phi_f[-1]\Adot) \ \cong \ \hyp^{k}(B_\epsilon(\mbf x)\cap X, F_{f, \mbf x};\ \Adot).
$$
Note that, if $\Adot$ is not the constant sheaf (or some shift of it), one must take care in thinking of the vanishing cycle cohomology as the ``reduced'' cohomology; it will not be as trivial as removing a copy of $\Z$ in one degree. Note, also, that we use the standard definition of the vanishing cycles; the one used in \cite{kashsch} is shifted by one from this.

The costalk cohomology of the nearby cycles gives the cohomology of the Milnor fiber modulo its boundary; that is, for all $\mbf x\in V(f)$, 
$$
H^k({\hat m}_{\mbf x}^!\psi_f[-1]\Adot)_{\mbf x} \ \cong \ \hyp^{k-1}(F_{f, \mbf x}, \partial F_{f, \mbf x};\ \Adot).
$$

There is a natural map between functors ${\hat m}_{\mbf x}^!\arrow{\tau_{\mbf x}}{\hat m}_{\mbf x}^*$. On cohomology, this yields the map induced by the inclusion of pairs $(F_{f, \mbf x}, \emptyset)\hookrightarrow (F_{f, \mbf x}, \partial F_{f, \mbf x})$:
$$
\hyp^{k-1}(F_{f, \mbf x}, \partial F_{f, \mbf x};\ \Adot) \ \rightarrow \ \hyp^{k-1}(F_{f, \mbf x};\ \Adot).
$$

There are monodromy automorphisms in the derived category, induced by letting the value of $f$ travel once, counterclockwise, around a small circle:
$$
T_f: \psi_f[-1]\Adot\rightarrow \psi_f[-1]\Adot\hskip .3in\textnormal{and}\hskip 0.3in \widetilde T_f: \phi_f[-1]\Adot\rightarrow \phi_f[-1]\Adot.
$$
On the stalk cohomology of the nearby cycles, ${\hat m}_{\mbf x}^*T_f$ induces the usual monodromy automorphism on $\hyp^*(F_{f, \mbf x};\ \Adot)$. On the costalk cohomology of the nearby cycles, ${\hat m}_{\mbf x}^!T_f$ induces the usual monodromy automorphism on $\hyp^*(F_{f, \mbf x},\partial F_{f,\mbf x};\ \Adot)$.

The following theorem is immediate from \cite{singenrich}, Theorems 3.3 and  3.5.

\begin{thm} If the Morse modules of $\Adot$ are all free Abelian, then so are all of the Morse modules of $\psi_f[-1]\Adot$ and $\phi_f[-1]\Adot$.
\end{thm}

\bigskip

There are two standard natural distinguished triangles, and their duals, related to the Milnor fibration. We write these triangles in their in-line forms. We remind the reader that $j_f$ denotes the closed inclusion of $V(f)$ into $X$, and we let $i_f$ denote the open inclusion of $X-V(f)$ into $X$.

\medskip

\noindent {\bf The vanishing triangle}:
$$
\rightarrow j_f^*[-1]\Adot\xrightarrow{\operatorname{comp}}\psi_f[-1]\Adot\xrightarrow{\operatorname{can}}\phi_f[-1]\Adot\arrow{[1]}.
$$

\noindent {\bf The Wang triangle}:
$$
\rightarrow j_f^*[-1]{i_f}_*i_f^*\Adot\rightarrow\psi_f[-1]\Adot\xrightarrow{\operatorname{id}-T_f}\psi_f[-1]\Adot\arrow{[1]}.
$$

\noindent {\bf The dual vanishing triangle}:
$$
\rightarrow \phi_f[-1]\Adot\xrightarrow{\operatorname{var}}\psi_f[-1]\Adot\xrightarrow{\operatorname{pmoc}} j_f^![1]\Adot\arrow{[1]}.
$$

\noindent {\bf The dual Wang triangle}:
$$
\rightarrow\psi_f[-1]\Adot\xrightarrow{\operatorname{id}-T_f}\psi_f[-1]\Adot\rightarrow j_f^![1]{i_f}_!i_f^!\Adot\arrow{[1]}.
$$

There are equalities of natural maps 
$$\operatorname{var}\circ\operatorname{can}=\operatorname{id}-T_f\hskip 0.4in\textnormal{and}\hskip 0.4in \operatorname{can}\circ\operatorname{var}=\operatorname{id}-\widetilde T_f.
$$

\bigskip

The following theorem is a dual version of a generalization of the primary result of L\^e in \cite{leattach}.

\begin{thm}\label{thm:sliceattach} Suppose that $\Pdot$ is a perverse sheaf on $X$, with free Abeian Morse modules. Let $\mbf x\in V(f)$, and let $L$ be the restriction to $X$ of a generic affine linear form such that $L(\mbf x)=0$. Let $\check m_{\mbf x}$ denote the inclusion of $\{\mbf x\}$ into $V(f, L)$, let $\hat m_{\mbf x}$ denote the inclusion of $\{\mbf x\}$ into $V(f)$, let $\hat L$ denote the restriction of $L$ to $V(f)$, and let $j_{\hat L}$ denote the inclusion of $V(f, L)$ into $V(f)$.

Then, for all $k\geq 1$,
$$H^{k+1}\left(\hat m^!_{\mbf x}\psi_f[-1]\Pdot\right) \ \cong \ H^{k}\left(\check m^!_{\mbf x}\psi_{f_{|_{V(L)}}}[-1]\big(\Pdot_{|_{V(L)}}[-1]\big)\right).
$$
In addition, there is an exact sequence
$$
0\rightarrow H^{0}\left(\hat m^!_{\mbf x}\psi_f[-1]\Pdot\right)\rightarrow \Z^\tau\rightarrow H^{0}\left(\check m^!_{\mbf x}\psi_{f_{|_{V(L)}}}[-1]\big(\Pdot_{|_{V(L)}}[-1]\big)\right)\rightarrow H^{1}\left(\hat m^!_{\mbf x}\psi_f[-1]\Pdot\right)\rightarrow 0,
$$
where $\tau$ can be calculated as the intersection number of a general version of the relative polar curve with $V(f)$ at $\mbf x$.

More generally, if $H^{c}$ denotes the intersection of $X$ with a generic affine linear subspace of codimension $c$ in $\U$, then,  
for all $k\geq 1$
$$H^{k+c}(\hat m^!_{\mbf x}\psi_f[-1]\Pdot)\ \cong \ H^{k}\left(\check m^!_{\mbf x}\psi_{f_{|_{H^{c}}}}[-1]\big(\Pdot_{|_{H^{c}}}[-c]\big)\right).
$$
\end{thm}

\begin{proof} Theorem 4.1 of \cite{enrichpolar} implies that we have an isomorphism 
$$H^*\left(\check m^!_{\mbf x}\psi_{\hat L}[-1]\psi_f[-1]\Pdot\right) \ \cong \ H^*\left(\check m^!_{\mbf x}\psi_{f_{|_{V(L)}}}[-1]\big(\Pdot_{|_{V(L)}}[-1]\big)\right).\eqno{{(\dagger)}}
$$
The dual vanishing triangle for $\hat L$, applied to $\psi_f[-1]\Pdot$, yields
$$
\rightarrow \phi_{\hat L}[-1]\psi_f[-1]\Pdot\rightarrow \psi_{\hat L}[-1]\psi_f[-1]\Pdot\rightarrow j_{\hat L}^![1]\psi_f[-1]\Pdot\arrow{[1]}.$$
Applying $\check m^!_{\mbf x}$, we have a distinguished triangle
$$
\rightarrow \check m^!_{\mbf x}\phi_{\hat L}[-1]\psi_f[-1]\Pdot\rightarrow \check m^!_{\mbf x}\psi_{\hat L}[-1]\psi_f[-1]\Pdot\rightarrow \hat m^!_{\mbf x}[1]\psi_f[-1]\Pdot\arrow{[1]}.\eqno{{(*)}}$$

Now, for generic $L$, $\mbf x$ is an isolated point in the support of the perverse sheaf $\phi_{\hat L}[-1]\psi_f[-1]\Pdot$. Hence, its cohomology is concentrated in degree $0$, and Corollary 4.6.3 of \cite{enrichpolar} tells us that 
$$H^0\left(\check m^!_{\mbf x}\phi_{\hat L}[-1]\psi_f[-1]\Pdot\right) \ \cong \ H^0\left(\check m^*_{\mbf x}\phi_{\hat L}[-1]\psi_f[-1]\Pdot\right)\eqno{{(\ddagger)}}
$$
is $\Z^\tau$, where $\tau$ is intersection number of a general version of the relative polar curve with $V(f)$ at $\mbf x$.

The first two statements of the theorem now follow from taking the long exact cohomology sequence of $(*)$ our last distinguished triangle, and using $(\dagger)$ and $(\ddagger)$. The final statement follows from the first statement by induction. 
\end{proof}

\medskip

\noindent\rule{1in}{1pt}

\noindent{\bf $\Adot$-critical points}

\medskip

We want critical points, with respect to a complex $\Adot$, to be places where the hypercohomology of the fibers of $f$ changes.

\smallskip

\begin{defn} Suppose that $\Adot\in D^b_c(X)$. We say that $\mbf p\in X$ is an {\bf $\Adot$-critical point} of $f$ if and only if $H^*(\phi_{f-f(\mbf p)}\Adot)\neq 0$. We let $\Sigma_{\Adot}f$ denote the set of $\Adot$-critical points of $f$.

In the special case where $\Adot$ is the constant sheaf, or a shift of it, we simply write $\Sigma_{\Z}f$, rather than $\Sigma_{\Z^\bullet_X}f$.

\end{defn}

\bigskip

The following properties are well-known.

\begin{thm}
\begin{enumerate}

\item Suppose that $\strat$ is a Whitney stratification of $X$, with respect to which $\Adot$ is constructible. Then, $\Sigma_{\Adot}f$ is contained in the stratified critical locus, i.e., $\Sigma_{\Adot}f\subseteq \bigcup_{S\in\strat}\Sigma(f_{|_S})$.

\item As stratified critical values are locally isolated, $\Adot$-critical values are locally isolated, i.e., near a point $\mbf x\in X$, $\Sigma_{\Adot}f\subseteq V(f-f(\mbf x))$. Thus, near $\mbf x\in V(f-f(\mbf x))$, 
$$
\overline{\Sigma_{\Adot}f} \ = \ \supp \phi_{f-f(\mbf x)}\Adot.
$$ 

\item Suppose that $\Pdot$ is a perverse sheaf on $X$, $\mbf x\in V(f)$, and let $s:=\dim_{\mbf x}\Sigma_{\Pdot}f$.  Then, unless $-s\leq k\leq 0$, $H^k\left ({\hat m}^*_{\mbf x}\phi_f[-1]\Pdot\right )=0$ and, unless $0\leq k\leq s$, $H^k\left ({\hat m}^!_{\mbf x}\phi_f[-1]\Pdot\right )=0$.

\item Suppose that $\Pdot$ is a perverse sheaf on $X$, $\mbf x\in V(f)$, and let $L$ be the restriction to $X$ of a generic affine linear form such that $L(\mbf x)=0$. Then, $\dim_{\mbf x}\Sigma_{\Pdot}L =0$, and so 
$$
{\hat m}^*_{\mbf x}\phi_L[-1]\Pdot \ \cong \ {\hat m}^!_{\mbf x}\phi_L[-1]\Pdot
$$
\noindent has possibly non-zero cohomology only in degree $0$.

\end{enumerate}
\end{thm}

\medskip

Suppose that $X$ is a local complete intersection of pure dimension $(n+1)$, and $\Pdot=\Z^\bullet_X[n+1]$. Then,  (4) above is a weak form of the fact that the complex link of a point has the homotopy-type of a bouquet of $n$-spheres.

\smallskip

The next result is well-known in the classical situation. In the general case, it follows in the same way as the classical case: from the fact that $L$ may be chosen so that the intersection of $V(L)$ with the relative polar of $(f, L)$ with respect to $\Adot$ is an isolated point. See \cite{enrichpolar}.

\begin{prop} Let $\mbf x\in V(f)$, and let $L$ be the restriction to $X$ of a generic affine linear form such that $L(\mbf x)=0$. Then, in an open neighborhood of $\mbf x$,
$$
\left(\overline{\Sigma_{{\Adot}_{|_{V(L)}}}\left(f_{|_{V(L)}}\right)}\right)\backslash\{\mbf x\}\ \ =  \ \left(V(L)\cap \overline{\Sigma_{\Adot}f}\right)\backslash\{\mbf x\}.
$$
In particular, if $\dim_{\mbf x}\Sigma_{\Adot}f\geq 1$, then 
$$
\dim_{\mbf x}\Sigma_{{\Adot}_{|_{V(L)}}}\left(f_{|_{V(L)}}\right) \ =  \ -1+\dim_{\mbf x}\Sigma_{\Adot}f.
$$
\end{prop}

\medskip

\noindent\rule{1in}{1pt}

\noindent{\bf The variation map}

\medskip

The natural map $\phi_f[-1]\xrightarrow{\operatorname{var}}\psi_f[-1]$ in the dual vanishing triangle is called the {\it variation map}, and we will now explain its relationship to the classical variation map between the cohomology of the Milnor fiber and the Milnor fiber modulo its boundary.

\medskip

Suppose that $\mbf p\in V(f)$ and that $\dim_{\mbf p} \Sigma_{\Adot}f=0$. Then, since $\mbf p$ is an isolated point in the support of $\phi_f[-1]\Adot$, 
$${\hat m}^*_{\mbf p}\phi_f[-1]\Adot \ \cong  \ {\hat m}^!_{\mbf p}\phi_f[-1]\Adot,$$
and so we have a morphism $\gamma_{\mbf p}$ defined by
$$
{\hat m}^*_{\mbf p}\psi_f[-1]\Adot\xrightarrow{{\hat m}^*_{\mbf p}(\operatorname{can})}{\hat m}^*_{\mbf p}\phi_f[-1]\Adot \ \cong \ {\hat m}^!_{\mbf p}\phi_f[-1]\Adot\xrightarrow{{\hat m}^!_{\mbf p}(\operatorname{var})}{\hat m}^!_{\mbf p}\psi_f[-1]\Adot.
$$

Recalling the natural map ${\hat m}_{\mbf p}^!\psi_f[-1]\Adot\arrow{\tau_{\mbf p}}{\hat m}_{\mbf p}^*\psi_f[-1]\Adot$, we have
$$
\tau_{\mbf p}\circ\gamma_{\mbf p}=\operatorname{id}-{\hat m}_{\mbf p}^*T_f\hskip 0.2in\textnormal{and}\hskip 0.2in \gamma_{\mbf p}\circ\tau_{\mbf p}=\operatorname{id}-{\hat m}_{\mbf p}^!T_f.
$$

On the level of cohomology, this gives us: if $\mbf p\in V(f)$ and $\dim_{\mbf p} \Sigma_{\Adot}f=0$, then there are well-defined {\it variation maps}:
$$
\hyp^{k-1}(F_{f, \mbf p};\ \Adot)\xrightarrow{\nu_{\mbf p}^{k-1}}\hyp^{k-1}(F_{f, \mbf p}, \partial F_{f, \mbf p};\ \Adot)
$$
such that, given the canonical maps
$$
\hyp^{k-1}(F_{f, \mbf p},\partial F_{f, \mbf p};\ \Adot)\xrightarrow{\tau_{\mbf p}^{k-1}}\hyp^{k-1}(F_{f, \mbf p} ;\ \Adot),
$$
we have
$$
\hyp^{k-1}(F_{f, \mbf p},\partial F_{f, \mbf p};\ \Adot)\xrightarrow{\nu_{\mbf p}^{k-1}\circ\tau_{\mbf p}^{k-1} \ = \ \operatorname{id}-T^{k-1}_{f, \mbf p}}\hyp^{k-1}(F_{f, \mbf p},\partial F_{f, \mbf p} ;\ \Adot),
$$
and
$$
\hyp^{k-1}(F_{f, \mbf p};\ \Adot)\xrightarrow{\tau_{\mbf p}^{k-1}\circ\nu_{\mbf p}^{k-1} \ = \ \operatorname{id}-T^{k-1}_{f, \mbf p}}\hyp^{k-1}(F_{f, \mbf p};\ \Adot).
$$

\section{Results for local complete intersections} 

Our goal in this section is to see what the derived category results from the previous section tell us about the classical objects -- the Milnor fiber, the real link, and the complex link -- that we discussed in Section 2. All of the results which appear here are stated in completely classical terms, though the proofs go through the derived category.

Throughout this section, $X$ will denote a purely $(n+1)$-dimensional local complete intersection. For convenience, we assume that $\0\in X$,  that $X$ is locally embedded in an open subset $\U$ of $\C^N$, and that $\0$ is a singular point of $X$ (for otherwise, everything here follows trivially from known results). We let $f:(X,\mbf 0)\rightarrow(\C,0)$ be a complex analytic function. Finally, whenever we write ordinary cohomology, we mean with integral coefficients.

Note that, if $f$ vanishes identically on some irreducible components of $X$, the Milnor fiber of $f$ is the same as the Milnor fiber of $f$ restricted to $\overline{X\backslash V(f)}$. Consequently, we will assume that $f$ does not vanish identically on any irreducible component of $X$ which contains the origin.

In the proofs,  $\Pdot$ will denote the perverse sheaf $\Z^\bullet_X[n+1]$; by \thmref{thm:dprops}, $\vd\Pdot$  is perverse, with free Abelian Morse modules, and is isomorphic to $\Z_{X\backslash\Sigma X}^\bullet[n+1]$ when restricted to $X\backslash\Sigma X$.

As in the introduction, we let $d:=\dim_\0\Sigma X$. If $d=0$, we define $d_f:=0$; otherwise, we define $d_f:=\dim_\0\overline{\Sigma X\backslash V(f)}$. We let $s_f:=\dim_\0\Sigma_{\Z}f$. If $\0\not\in \Sigma_{\Z}f$, we set $s_f=-\infty$.

\medskip

\begin{thm}\label{thm:onedegree} The reduced cohomology $\widetilde H^k(F_{f, \mbf 0})=0$, except, possibly, when $n-s_f\leq k\leq n$. Furthermore, $\widetilde H^{n-s_f}(F_{f, \mbf 0})$ is free Abelian.

Suppose that we are in the special case where $s_f=0$. Then, the reduced cohomology $\widetilde H^k(F_{f, \mbf 0})$ equals zero, except, possibly, when $k=n$. In addition, $\operatorname{rank}\widetilde H^n(F_{f, \mbf 0})\geq \operatorname{rank}\widetilde H^n(\cL_{X,\mbf 0})$.
\end{thm}

\begin{proof} The condition that $s_f=\dim_{\mbf 0}\Sigma_{\Z}f$ is equivalent to $s_f=\dim_\0\operatorname{supp} \phi_f[-1]\Pdot$. As $\phi_f[-1]\Pdot$ is perverse, this implies that the only possibly non-zero stalk cohomology is in degrees $j$ such that  $-s_f\leq j\leq 0$, i.e., if $j$ is not in this range, then
$$
H^j(\phi_f[-1]\Z^\bullet_X[n+1])_\0 \ \cong  \ H^{j+n}(\phi_f\Z^\bullet_X)_\0  \ \cong  \ \widetilde H^{j+n}(F_{f, \0}) \ =  \ 0.
$$

L\^e's result tells us not just that $\Pdot$ is a perverse, but also that all of the Morse modules are free Abelian. Consequently, the Morse modules of $\phi_f[-1]\Pdot$ are free Abelian. It follows that $\widetilde H^{n-s_f}(F_{f, \mbf 0})$ is free Abelian. Finally, that the rank of $H^n(F_{f, \mbf 0})$ is at least the rank of $\widetilde H^n(\cL_{X,\mbf 0})$ is immediate from Theorem 5.3 of \cite{micromorse}.
\end{proof}

\begin{rem} The first statement above can be found in  Example 6.0.12 of \cite{schurbook}. The second statement also follows from Sch\"urmann's work in Chapter 6 of  \cite{schurbook}.

The result from Theorem 5.3 of \cite{micromorse} is more general and more precise than the statement in \thmref{thm:onedegree}, but, of course, that makes the statement harder to read. 

There are also similar, more refined, results on the level of homotopy-type by Siersma \cite{siersmabouquet} and Tib\u ar \cite{tibarbouq}, under the strictly stronger hypothesis that
$f$ has a {\bf stratified} isolated critical point $\0$. 

\end{rem}

\bigskip

\begin{thm}\label{thm:bignew} $H^k(F_{f,\0}, \partial F_{f,\0})=0$ unless $k=2n$ or $n\leq k\leq n+\operatorname{max}\{s_f, d_f\}$. In addition, $H^k(F_{f,\0}, \partial F_{f,\0})$ is free Abelian when $k=n$ and, if $d_f\neq n$, when $k=2n$.
\end{thm}
\begin{proof} Using notation from the previous section, we need to look at 
$$H^j(\hat m_\0^!\psi_f[-1]\Pdot) \ \cong \ H^{j+n}(F_{f,\0}, \partial F_{f,\0}),
$$
and show that it is $0$, unless $j=n$ or $0\leq j\leq \operatorname{max}\{s_f, d_f\}$. As $\psi_f[-1]\Pdot$ is supported on a set of dimension at most $n$, $H^j(\hat m_\0^!\psi_f[-1]\Pdot)$ unless $0\leq j\leq n$. Thus, what we need to show is that $H^j(\hat m_\0^!\psi_f[-1]\Pdot)=0$ if $\operatorname{max}\{s_f, d_f\}+1\leq j\leq n-1$.

\medskip

\noindent $\bullet$ $d_f=0$ case:

\smallskip

If $d_f=0$, then $F_{f,\0}$ is a smooth manifold with boundary, and Poincar\'e-Lefschetz duality tells us that $H^k(F_{f,\0}, \partial F_{f,\0})$ is isomorphic to $H_{2n-k}(F_{f,\0})$. The result now follows from \thmref{thm:onedegree} and the Universal Coefficient Theorem.

\medskip

\noindent $\bullet$ $d_f$ arbitrary case:

\smallskip

Let $H^{d_f}$ be the intersection of $X$ with a generic $d_f$-codimensional linear subspace in $\C^N$. Then, $H^{d_f}$ is a purely $(n+1-d_f)$-dimensional local complete intersection, and $g:=f_{|_{H^{d_f}}}$ is a function such that $s_g= \max\{0, s-d_f\}$ and $d_g= 0$. 

By the previous case, we have that $H^j(\hat m_\0^!\psi_g[-1]\big(\Pdot_{|_{H^{d_f}}}[-d_f]\big))=0$ if $\max\{0, s_f-d_f\}+1\leq j\leq n-d_f-1$. That $H^j(\hat m_\0^!\psi_f[-1]\Pdot)=0$ if $\operatorname{max}\{s_f, d_f\}+1\leq j\leq n-1$ now follows immediately from \thmref{thm:sliceattach}. That $H^n(F_{f,\0}, \partial F_{f,\0})$ is free Abelian also follows from \thmref{thm:sliceattach}, since $H^{0}\left(\hat m^!_{\mbf x}\psi_f[-1]\Pdot\right)$ injects into $\Z^\tau$. That $H^{2n}(F_{f,\0}, \partial F_{f,\0})$ is free Abelian also follows immediately from the $d_f=0$ case and the isomorphisms in \thmref{thm:sliceattach}, provided $n-d_f\geq 1$.
\end{proof}

\smallskip

If $L$ is the restriction of a generic linear form to $X$, then $s_L\leq 0$, $d_L=d$, and $F_{L,\0}$ is the complex link $\cL_{X, \0}$. Applying the theorem above, we immediately obtain:

\begin{cor}\label{cor:linkzero} The relative cohomology $H^k(\cL_{X, \0}, \partial \cL_{X, \0})=0$ unless $k=2n$ or $n\leq k\leq n+d$. In addition, $H^k(\cL_{X, \0}, \partial \cL_{X, \0})$ is free Abelian when $k=n$ and, if $d\neq n$, when $k=2n$.
\end{cor}

\bigskip

Taking the stalk cohomology of the Wang triangle and the costalk cohomology of the dual Wang triangle (see the previous section), both at the origin, and combining this with \thmref{thm:onedegree} and \thmref{thm:bignew},  we immediately obtain:

\begin{thm}\label{thm:wangs} There are long exact sequences, the Wang sequence and dual Wang sequence:
$$
\cdots\rightarrow H^{k-1}(F_{f,\0})\rightarrow H^k(K_{X,\0}\backslash K_{V(f),\0})\rightarrow  H^{k}(F_{f,\0})\xrightarrow{\operatorname{id}-T_{f, \0}}H^{k}(F_{f,\0})  \rightarrow H^{k+1}(K_{X,\0}\backslash K_{V(f),\0})\rightarrow\cdots
$$
and

$
\cdots\rightarrow H^{k-1}(F_{f,\0}, \partial F_{f,\0})\rightarrow H^k(K_{X,\0}, K_{V(f),\0})\rightarrow  H^{k}(F_{f,\0}, \partial F_{f,\0})\xrightarrow{\operatorname{id}-T_{f, \0}} \hfill
$

$
\hfill H^{k}(F_{f,\0}, \partial F_{f,\0})  \rightarrow H^{k+1}(K_{X,\0}, K_{V(f),\0})\rightarrow\cdots$.
\end{thm}

\bigskip

\begin{rem} Note that, if $X$ itself has a non-isolated singularity at $\0$, then $K_{X,\0}$ and $F_{f,\0}$ are not manifolds, and the second long exact sequence above is {\bf not} obtained via Poincar\'e-Lefschetz duality from the first one.

\end{rem}

From \thmref{thm:wangs} and \thmref{thm:bignew}, we immediately conclude:

\begin{cor} The integral cohomology $H^k(K_{X,\0}\backslash K_{V(f),\0})=0$ except, possibly, when $k=0, 1$ or $n-s_f\leq k\leq n+1$. 

The integral relative cohomology $H^k(K_{X,\0}, K_{V(f),\0})=0$ except, possibly, when $k=2n, 2n+1$ or $n\leq k\leq n+ \operatorname{max}\{s_f, d_f\} +1$.
\end{cor}

By taking $f$ to be the restriction of a generic linear form, we obtain:

\begin{cor} Let $H$ denote e generic hyperplane in the ambient affine space. Then, the integral cohomology $H^k(K_{X,\0}\backslash K_{X\cap H,\0})=0$ except, possibly, when $k=0, 1$ or $n\leq k\leq n+1$. 

The integral relative cohomology $H^k(K_{X,\0}, K_{X\cap H,\0})=0$ except, possibly, when $k=2n, 2n+1$ or $n\leq k\leq n+ d+1$.
\end{cor}

\bigskip

\begin{thm}
Suppose that $n\geq 1$ and $s_f=0$. Let $\omega$ denote the map from $H^n(F_{f,\0}, \partial F_{f,\0})$ to $H^n(F_{f,\0})$ induced by the inclusion $(F_{f,\0}, \emptyset)\hookrightarrow (F_{f,\0}, \partial F_{f,\0})$.

Then:

\begin{enumerate}

\item There is well-defined variation map 
$$\operatorname{var}:H^n(F_{f,\0})\rightarrow H^n(F_{f,\0}, \partial F_{f,\0})
$$
such that $\operatorname{var}\circ\, \omega =\operatorname{id}-T_f$ on $H^n(F_{f,\0}, \partial F_{f,\0})$ and $\omega\circ \operatorname{var} =\operatorname{id}-T_f$ on $H^n(F_{f,\0})$.

\item
There is an exact sequence
$$
0\rightarrow H^n(K_{X,\0})\rightarrow H^n(F_{f,\0})\arrow{\operatorname{var}} H^n(F_{f,\0}, \partial F_{f,\0})\rightarrow H^{n+1}(K_{X,\0})\rightarrow 0.
$$
In particular, $H^n(K_{X,\0})$ injects into $H^n(F_{f,\0})$, and $\operatorname{var}$
  is an isomorphism if and only if $H^n(K_{X,\0})= H^{n+1}(K_{X,\0})=0$. 
  
\item For all $k\geq n+1$, $H^k(F_{f,\0}, \partial F_{f,\0})\cong H^{k+1}(K_{X,\0})$, and so, for all $k\geq n+1$, $H^{k}(\partial F_{f,\0}) \ \cong \ H^{k+2}(K_{X,\0})$.

\item For all $k\leq n-1$, $\widetilde H^{k}(K_{X,\0})=0$.
  
  \end{enumerate}
\end{thm}
\begin{proof} The existence of the variation map, satisfying Item (1), is immediate from our general categorical discussion in the previous section.

\bigskip

Now, since $s_f=0$ and $n\geq1$, the map $H^n(F_{f,\0})\xrightarrow{\operatorname{can}} \widetilde H^n(F_{f,\0})$ is an isomorphism. In addition, since $s_f=0$, 
$$
\widetilde H^n(F_{f,\0}) \ \cong \ H^0\left(\hat m^*_\0\phi_f[-1]\Pdot\right) \ \cong \ H^0\left(\hat m^!_\0\phi_f[-1]\Pdot\right), 
$$
and $H^k\left(\hat m^!_\0\phi_f[-1]\Pdot\right)=0$ for $k\neq 0$.

If we apply $\hat m_\0^!$ to the dual vanishing triangle, we obtain the distinguished triangle 
$$
\rightarrow \hat m_\0^!\phi_f[-1]\Pdot\xrightarrow{\operatorname{var}}\hat m_\0^!\psi_f[-1]\Pdot\xrightarrow{\operatorname{pmoc}}  m_\0^![1]\Pdot\arrow{[1]}.
$$
Taking the long exact cohomology sequence of this triangle immediately yields Items (2), (3), and (4) after one notes that
$$
H^k(m_\0^![1]\Pdot) \ \cong \ H^{k+n+2}(B_\epsilon(\0)\cap X, S_\epsilon(\0)\cap X) \ \cong \ \widetilde H^{k+n+1}(K_{X, \0}).
$$
\end{proof}

\bigskip

Applying this theorem to the restriction of a generic linear form, we conclude:

\begin{cor}\label{cor:main} Suppose that $n\geq 1$, that $L$ is the restriction of a generic linear form to $X$, and that $H=V(L)$. Then:

\begin{enumerate}

\item There is well-defined variation map 
$$\operatorname{var}:H^n(\cL_{X, \0})\rightarrow H^n(\cL_{X, \0}, \partial \cL_{X, \0})
$$
such that $\operatorname{var}\circ\, \omega =\operatorname{id}-T_L$ on $H^n(\cL_{X, \0}, \partial \cL_{X, \0})$ and $\omega\circ \operatorname{var} =\operatorname{id}-T_L$ on $H^n(\cL_{X, \0})$.

\item
There is an exact sequence
$$
0\rightarrow H^n(K_{X,\0})\rightarrow H^n(\cL_{X, \0})\arrow{\operatorname{var}} H^n(\cL_{X, \0}, \partial \cL_{X, \0})\rightarrow H^{n+1}(K_{X,\0})\rightarrow 0.
$$
In particular, $H^n(K_{X,\0})$ injects into $H^n(\cL_{X, \0})$, and $\operatorname{var}$
  is an isomorphism if and only if $H^n(K_{X,\0})= H^{n+1}(K_{X,\0})=0$. 
  
\item For all $k\geq n+1$, $H^k(\cL_{X, \0}, \partial \cL_{X, \0})\cong H^{k+1}(K_{X,\0})$, and so, for all $k\geq n+1$, 
$$H^k(K_{X\cap H, \0}) \ \cong \  H^k(\partial \cL_{X, \0}) \ \cong \  H^{k+2}(K_{X,\0}).$$
  
  \end{enumerate}
\end{cor}

\medskip

\begin{rem} That $H^n(K_{X,\0})$ injects into $H^n(\cL_{X, \0})$, regardless of the dimension of the singular set of $X$ generalizes part of the result given by Dimca in Proposition 6.1.22 of \cite{dimcasheaves}, in which $d$ is required to be $0$.
\end{rem}

Combining Item (3) of the previous corollary with \corref{cor:linkzero}, we obtain:

\begin{cor} The integral cohomology $H^{k}(K_{X,\0})=0$, unless $k=0$, $k=2n+1$, or $n\leq k\leq n+d+1$.
\end{cor}

\bigskip

\begin{exm}
Of course, it is natural to ask for an example where $d>0$, $n+d+1\neq 2n+1$ (i.e., $d\neq n$), and $H^{n+d+1}(K_{X,\0})\neq 0$.

In fact, it is easy to produce such examples. Let $f:\C^5\rightarrow\C$ be given by $f(t,w,x,y,z)=w^2+x^2+y^2+z^2$ and let $X=V(f)$, i.e., $X$ is the product of an ordinary quadratic singularity in $\C^4$ with a complex line. Here, $n=3$ and $d=1$.

Then, by Item (3) of \corref{cor:main}, and since $\cL_{X, \0}$ is contractible,
$$
H^{n+d+1}(K_{X,\0}) \ \cong \ H^{5}(K_{X,\0})  \ \cong \ H^4(\cL_{X, \0}, \partial \cL_{X, \0}) \ \cong  \ H^3( \partial \cL_{X, \0}) \ \cong  \ H^3(K_{X\cap V(t),\0}).
$$

But $K_{X\cap V(t), \0}$ is the real link of the ordinary quadratic singularity $V(w^2+x^2+y^2+z^2)$ in $\C^4$. It is well-known that the monodromy on $H_3$ of the corresponding Milnor fiber is the identity (since the Lefschetz number of the monodromy must be zero). Thus, using the Wang sequence on homology, we know that
$$
\Z \ \cong \ H_3(S^7_\epsilon\backslash K_{X\cap V(t),\0}) \ \cong \ H^4(S^7_\epsilon, K_{X\cap V(t),\0}) \ \cong \ H^3(K_{X\cap V(t),\0}) \ \cong \ H^{5}(K_{X,\0}).
$$
\end{exm}
\newpage
\bibliographystyle{plain}
\bibliography{Masseybib}
\end{document}